\documentclass[11pt]{article}
\usepackage[numbers,sort&compress]{natbib}
\usepackage{enumerate}
\usepackage{amscd}
\usepackage{amsmath}
\usepackage{latexsym}
\usepackage{amsfonts}
\usepackage{amssymb}
\usepackage{amsthm}
\usepackage{verbatim}
\usepackage{mathrsfs}
\usepackage{enumerate}
\usepackage{hyperref}

\theoremstyle{plain}
\theoremstyle{definition}\newtheorem{theorem}{Theorem}[section]
\theoremstyle{plain}\newtheorem{lemma}[theorem]{Lemma}
\theoremstyle{plain}\newtheorem{coro}[theorem]{Corollary}
\theoremstyle{plain}
\theoremstyle{remark}\newtheorem{remark}{Remark}[section]
\usepackage{xcolor}

\newcommand{\B}{\Big}

\newcommand{\be}{\begin{equation}}
\newcommand{\ee}{\end{equation}}
\newcommand{\ba}{\begin{aligned}}
	\newcommand{\ea}{\end{aligned}}

\newcommand{\f}{\frac}

\newcommand{\ben}{\begin{enumerate}}
	\newcommand{\een}{\end{enumerate}}

\newcommand{\Rmnum}[1]{\expandafter\@slowromancap\romannumeral #1@}

\allowdisplaybreaks

\numberwithin{equation}{section}
\begin{document}
	\title{ Refined blow up criteria for the full compressible Navier-Stokes equations involving temperature  }
	\author{Quanse Jiu\footnote{
School of Mathematical Sciences, Capital Normal University, Beijing 100048, PR China. Email: jiuqs@cnu.edu.cn},\,\,\,\,Yanqing Wang\footnote{ Department of Mathematics and Information Science, Zhengzhou University of Light Industry, Zhengzhou, Henan  450002,  P. R. China Email: wangyanqing20056@gmail.com}\;~ and\,
	Yulin Ye\footnote{School of Mathematics and Statistics,
Henan University,
Kaifeng, 475004,
P. R. China Email: ylye@vip.henu.edu.cn   }}
\date{}
\maketitle
\begin{abstract}
In this paper, inspired by the study of the energy flux in local energy inequality of the 3D  incompressible Navier-Stokes equations, we improve almost all the blow up criteria
involving temperature to allow  the temperature  in its scaling  invariant  space  for  the 3D full compressible Navier-Stokes equations.
Enlightening  regular  criteria via pressure $\Pi=\f{\text {divdiv}}{-\Delta}(u_{i}u_{j})$ of the 3D  incompressible Navier-Stokes equations on bounded domain,
we generalize   Beirao da Veiga's result in \cite{[da Veiga95]} from the incompressible Navier-Stokes equations
  to the isentropic compressible Navier-Stokes system in   the case away from vacuum.

  \end{abstract}
	\noindent {\bf MSC(2000):}\quad 35B65, 35D30, 76D05 \\\noindent
	{\bf Keywords:} full compressible Navier-Stokes equations;   strong solutions; blow-up criteria
	\section{Introduction}
	\label{intro}
	\setcounter{section}{1}\setcounter{equation}{0}
We study the following  system of  Newton heat-conducting compressible fluid in three-dimensional space
\be\left\{\ba\label{FNS}
	&\rho_t+\nabla \cdot (\rho u)=0, \\
&\rho  u_t+\rho u\cdot\nabla u+\nabla
P(\rho,\theta)-\mu\Delta u-(\mu+\lambda)\nabla\text{div\,}u=0,\\
& c_{v}[\rho \theta_t+\rho u\cdot\nabla\theta] +P\text{div\,}u-\kappa\Delta\theta=\frac{\mu}{2}\left|\nabla u+(\nabla u)^{\text{tr}}\right|^2+\lambda(\text{div\,}u)^2,\\&(\rho, u, \theta)|_{t=0}=(\rho_0, u_0, \theta_0),
	\ea\right.\ee
 where $\rho,\, u,\, \theta$ stand for the flow density, velocity and the absolute temperature, respectively. The scalar function $P$ represents the   pressure, the state equation of which is  determined by
 \be
 P=R\rho\theta,  R>0,
 \ee
 and $\kappa$ is a positive constant.   $\mu$ and $\lambda$ are the coefficients of viscosity, which are assumed to be constants, satisfying the following physical restrictions:
\be\label{nares}
\mu>0,\ 2\mu+3\lambda\ge0.\ee  The
initial
conditions satisfy
\be\label{non-boundary}
\rho(x,t)\rightarrow0,\ u(x,t)\rightarrow0,\ \theta(x,t)\rightarrow0,\ \mathrm{as}\ |x|\rightarrow\infty,\ \mathrm{for}\ t\ge0.
\ee
Note that if  the triplet $(\rho(x,t),u(x, t),\theta(x, t) )  $ solves system \eqref{FNS}, then the triplet $(\rho_{\lambda},u_{\lambda},\theta_{\lambda})  $ is also a solution of \eqref{FNS} for any $\lambda\in R^{+},$ where
\be\label{eqscaling}
\rho_{\lambda}=  \rho(\lambda^{2}t,\lambda x),~~~~~u_{\lambda}=\lambda u(\lambda^{2}t,\lambda x),~~~~~\theta_{\lambda}=\lambda^{2} \theta(\lambda^{2}t,\lambda x).
\ee
There have been huge literatures on well-posedness of solutions to compressible Navier-Stokes equations, we only give a brief survey here. For the isentropic case, The first major breakthrough was made by P.L.Lions in \cite{lions}, where he first gave the global existence of weak solutions to the compressible Navier-Stokes equations when the constant $\gamma\geq \f{3N}{N+2}$ for $N=2$ or $3$. Then Feireisl et al.in \cite{FNP} improved the Lions'work to $\gamma >\f{3}{2}$ for $N=3$. In \cite{JZ}, Jiang and Zhang considered the spherical symmetric initial data and relaxed the restriction on $\gamma$ to the case $\gamma>1$. Huang, Li and Xin in \cite{hlx} obtained the global existence of classical solutions provided the initial energy is sufficiently small but the oscillation can be large. When the shear viscosity coefficient $\mu=costant>0$ and bulk viscosity satisfies $\lambda(\rho)=\rho^\beta$, Vaigant-Kazhikhov in \cite{VK} showed the two-dimensional system admits a unique strong solution in the periodic domain when $\beta>3$, it's emphasized that the initial data contains no vacuum and can be arbitrarily large. For the case involving heat conductivity, Feireisl in \cite{Feireisl} got the existence of variational solutions when the dimension  $N\geq 2$. It is noted that this is the very first attempt work in global existence of weak solutions for full compressible Navier-Stokes equations in high dimensions. Matsumura-Nishida in \cite{MN} obtained the global classical solution for initial data close to a non-vacuum equilibrium in some Sobolev space $H^s$. Later, Hoff in \cite{Hoff1995} considered the discontinuous case.
In \cite{[CK]}, the local strong solutions of equations \eqref{FNS} with initial data containing vacuum was established by Cho and   Kim (for details, see Theorem \ref{localwith vacuum} in section 2).
On the other hand, when the initial data contain vacuums, finite time blow-up of smooth solutions to the compressible Navier-Stokes system was discussed by Xin
\cite{[Xin]}, Xin and Yan \cite{[XY]} and Jiu, Wang and Xin \cite{[JWX]}. Since then a number of papers have been devoted to the study of blow up mechanism of strong solutions mentioned above  in \eqref{FNS}  and many blow up criteria are established (see for example, \cite{[DW],[FJO],[WZ13],[WZ17],[LXZ],[WL],[CY],[CCK],[HLX],[FJ],
[HL],[SWZ1],[SWZ2],[SZ],[HX]} and references therein).  
In particular, we list some works where vacuum is included as follows:

 Suppose that   $0<T^{\ast}<\infty$
 is the maximal time of existence of a strong  solution of system \eqref{FNS}.

\noindent Fan, Jiang and Ou  \cite{[FJO]}
\be\label{FJO}
\lim\sup\limits_{t\nearrow
T^\star}\left(\|\nabla
u\|_{L^1(0,t;L^\infty)}+\|\theta\|_{L^\infty(0,t;L^\infty)}\right)=\infty,~~~(\lambda< 7\mu);
\ee
Wen and Zhu  \cite{[WZ13]},
\be\label{wzadv}
\lim\sup\limits_{t\nearrow
T^\star}\left(\|\rho\|_{L^\infty(0,t;L^\infty)}+\|\theta\|_{L^\infty(0,t;L^\infty)}\right)=\infty, ~~~(\lambda<3\mu);
\ee
Huang, Li and Wang \cite{[HLW]}
\be \label{HLW}
\lim\sup\limits_{t\nearrow
T^\star}\left(\|\text{div\,}u\|_{L^1(0,t;L^\infty)}+\|u\|_{L^p(0,t;L^q)}\right)=\infty, ~~\f2p+\f3q=1,~q>3;\ee
Huang  and Li   \cite{[HL2]}
\be \label{HL2}
\lim\sup\limits_{t\nearrow
T^\star}\left(\|\rho\|_{L^\infty(0,t;L^\infty)}+\|u\|_{L^p(0,t;L^q)}\right)=\infty, ~~\f2p+\f3q=1,~q>3;\ee
Li, Xu and  Zhu  \cite{[LXZ]}
\be\label{li}
\lim\sup\limits_{t\nearrow
T^\star}\left(\|\rho\|_{L^\infty(0,t;L^\infty)}+\|P\|_{L^\infty(0,t;L^\infty)}\right)=\infty, ~~~(\lambda<3\mu,~\kappa=0);
\ee
Wen and Zhu  \cite{[WZ17]},
\be\label{wzsiam}
\lim\sup\limits_{t\nearrow
T^\star}\left(\|\rho\|_{L^\infty(0,t;L^\infty)}+\|\rho\theta\|_{L^4(0,t;L^\f{12}{5})}\right)=\infty, ~~~(\lambda<3\mu);
\ee
Wang and Li \cite{[WL]}
\be\label{wl}
\lim\sup\limits_{t\nearrow
T^\star}\left(\|\text{div\,}u\|_{L^2(0,t;L^\infty)}+\|\theta\|_{L^\alpha(0,t;L^\beta)}\right)=\infty, ~~\f3\alpha+\f2\beta\geq2,~ \f1\alpha+\f2\beta\leq1,~1\leq\alpha\leq2,~\beta\geq4;
\ee
Choe and Yang \cite{[CY]}
\be
\lim\sup\limits_{t\nearrow
T^\star}\left(\|\rho\|_{L^\infty(0,t;L^\delta)}+\|\text{div\,} u\|_{L^\infty(0,t;L^3)}+\|\Delta\theta\|_{L^\infty(0,t;L^2)}\right)=\infty, ~~\text{for some} ~~ \delta\in(1,\infty).\ee
 The interesting of \eqref{HLW} and \eqref{HL2} is that they are independent of the temperature and they are in scaling invariant norm in the sense of  \eqref{eqscaling}. From \eqref{eqscaling},  the natural candidate  invariant spaces of temperature $\theta$ is ${L^{q}(0,T;L^{q})}$ with $\f2p+\f3q=2$. Therefore, a natural question is whether one can show blow up criteria for the full compressible
Navier-Stokes equations involving temperature in its scaling invariant space. The first objective of this paper is to address this issue and we obtain
\begin{theorem}
\label{th1.1}
Suppose $(\rho,u,\theta)$ is the unique strong solution in Theorem \ref{localwith vacuum}   and $\lambda<3\mu$.
If the maximal existence time $T^*$ is finite, then there holds
\be\limsup_{t\nearrow T^*}
  \|\rho\|_{L^{\infty}(0,t;L^{\infty})}+\|\theta\|_{L^{p}(0,t;L^{q})}= \infty,\label{wy1}\ee
  for some $p, q$ satisfying
    $$\f{2}{p}+\f{3}{q}=2,\ \  q>\f{3}{2}.$$
\end{theorem}

\begin{remark}Note that
\eqref{wy1}  can be replaced by
$$\limsup_{t\nearrow T^*}
  \|\text{div\,}u\|_{L^{1}(0,t;L^{\infty})}+\|\theta\|_{L^{p}(0,t;L^{q})}= \infty,$$
   or
  $$\limsup_{t\nearrow T^*}
  \|\nabla{u}\|_{L^{1}(0,t;L^{\infty})}+\|\theta\|_{L^{p}(0,t;L^{q})}= \infty,$$
  which  improves the known blow up criteria \eqref{FJO}.
\end{remark}
\begin{remark}
Theorem \ref{th1.1}  is an extension  of corresponding results in
  \eqref{wzadv}, \eqref{wzsiam}, \eqref{wl}, \eqref{fj} and \eqref{swz1}.
\end{remark}

We give some comments on the proof of   Theorem \ref{th1.1}.
The proof is motivated by the investigation of   regularity of suitable weak solutions to the 3D incompressible Navier-Stokes equations.
Suitable weak solutions  originated in pioneering work by Scheffer \cite{[Scheffer]}
 and in the celebrated paper by Caffarelli,  Kohn and  Nirenberg \cite{[CKN]} obey the local energy inequality.
  Roughly speaking, the energy flux in local energy inequality is  $\int_0^T\int |u|^{3}dxdt$, which can be bounded by (see eg. \cite{[GKT],[HWZ],[WWZ]})
\be\label{energy flux}
\int_0^T\int |u|^{3}dxdt\leq C  \B(\|u\|_{L^{\infty}L^{2}}^{2}+\|\nabla u\|_{L^{2}L^{2}}^{2}\B)\|u\|_{ L^{p}L^{q}}, \, ~~\f{2}{p}+\f{3}{q}=2.
\ee
We would like to mention that the inequality   \eqref{energy flux} plays an important role in the proof of results in  \cite{[GKT],[HWZ],[WWZ]}.

We turn our attentions back to the
 3D  compressible Navier-Stokes equations \eqref{FNS}. Under the hypothesis $\|\rho\|_{L^{\infty}L^{\infty}}$ and $\lambda<3\mu$,  we observe that there holds the following energy estimate to system \eqref{FNS}
$$\ba
& \f{d}{dt}\int\B[\f\mu2|\nabla u|^{2}+(\mu+\lambda)(\text{div\,}u  )^2+\f{1}{2\mu+\lambda}P^2 -2P\text{div\,}u +\f{C_3 C_\nu}{2}\rho\theta^2 +\f{C_4 +1}{2\mu}\rho| u|^4\B]\\
&+ \kappa\int|\nabla \theta|^{2}+\f{1}{2}\int \rho |\dot{u}|^{2}+\int
|u|^2\big|\nabla u\big|^2\leq C  \int \rho |\theta|^{3}  +C \int\rho |u|^{2}|\theta|^{2},
\ea$$
where $\f\mu2|\nabla u|^{2}+(\mu+\lambda)(\text{div\,}u  )^2+\f{1}{2\mu+\lambda}P^2 -2P\text{div\,}u +\f{C_3 C_\nu}{2}\rho\theta^2 +\f{C_4 +1}{2\mu}\rho| u|^4\geq \f\mu2|\nabla u|^{2}+\rho\theta^2+\f{C_4 +1}{2\mu}\rho| u|^4>0$ provided that the positive constant $C_3$ is suitably large. The key point is that the two terms of  right hand side in  the preceding inequality  are parallel to \eqref{energy flux}.  This helps us to prove Theorem \ref{th1.1}.

Without the restriction  $\lambda<3\mu$, we have
\begin{theorem}
\label{the1.2}
Suppose $(\rho,u,\theta)$ is the unique strong solution in Theorem \ref{localwith vacuum}.
If the maximal existence time $T^*$ is finite, then  one of the following results holds, for $p,q$ meeting
  $$\f{2}{p}+\f{3}{q}=2,~q>\f{3}{2},$$
 \begin{enumerate}[(1)]
 \item
 \be\label{nores1}\limsup_{t\nearrow T^*}\B(
  \|\rho\|_{L^{\infty}(0,t;L^{\infty})}+ \|\text{div\,}u \|_{L^{2}(0,t;L^{3})}+\|\theta\|_{L^{p}(0,t;L^{q})}\B)= \infty;\ee
  \item
   \be\label{nores2}\limsup_{t\nearrow T^*}\B( \|\text{div\,}u \|_{L^{1}(0,t;L^{\infty})}+ \|\text{div\,}u \|_{L^{4}(0,t;L^{2})}+\|\theta\|_{L^{p}(0,t;L^{q})}\B)= \infty;\ee
 \item   \be\label{nores3}\limsup_{t\nearrow T^*}\B( \|\text{div\,}u \|_{L^{2}(0,t;L^{\infty})}+\|\theta\|_{L^{p}(0,t;L^{q})}\B)= \infty.\ee
 \end{enumerate}
\end{theorem}
\begin{remark}
One can replace $\|\rho \|_{L^{\infty}(0,t;L^{\infty})}$ in \eqref{nores1}  by $  \|\text{div\,}u \|_{L^{1}(0,t;L^{\infty})}$,   which improves the known blow up criteria  \eqref{wl}.
\end{remark}

Though \eqref{nores1} and \eqref{nores2}  involve all the quantities in equations \eqref{FNS}, they are in scaling invariant spaces in the sense of \eqref{eqscaling}. We explain the motivation of \eqref{nores1} and \eqref{nores2} below. It is  known that the velocity $u$ (Serrin type), gradient  $\nabla u$ (Beirao da Veiga type), vorticity curl$u$, or pressure $\Pi=\f{\text {divdiv}}{-\Delta}(u_{i}u_{j})$ in scaling invariant norms
guarantee the regularity of the Leray-Hopf weak solutions to the 3D incompressible Navier-Stokes equations   (see eg. \cite{[GKT],[da Veiga95],[CKL07],[Serrin],[BG],[LX],[KL1],[KL2],[Zhou1],[Zhou2]}).
Serrin type criteria for the isentropic compressible fluid were proved by Huang, Li and Xin in \cite{[HLX]}. However, to the knowledge of the authors,
even though for the  isentropic compressible fluid in the presence of vacuum,
the following blow up criteria are unknown
 \be\label{de v}\limsup_{t\nearrow T^*}\B(
  \|\text{div\,}u \|_{L^{1}(0,t;L^{\infty})}+ \|\nabla u \|_{L^{p}(0,t;L^{q})}\B)= \infty,~~{\text with}~~\f{2}{p}+\f{3}{q}=2,~ q>3.\ee
The other case in \eqref{de v} $\f{3}{2}<q<3$ can be derived from the result \cite{[HLX]}   and Sobolev inequality and $q=3$ can be derived by  a slight variant of the proof of \cite{[HLX]}.
Hence, it seems that \eqref{nores1} and \eqref{nores2} without $\theta$ are still new results to the isentropic compressible fluid.
For the general case  \eqref{de v}, we can prove it for the strong solutions of the isentropic compressible Navier-Stokes equations in the case away from the vacuum. Before we state the result, we recall the known blow up criteria for the strong solutions of system \eqref{FNS} without vacuum.

\noindent Fan and Jiang \cite{[FJ]}
\be\label{fj}
\lim\sup\limits_{t\nearrow
T^\star}\left(\|(\rho,\f{1}{\rho},\theta)\|_{L^\infty(0,t;L^\infty)} +\|\rho\|_{L^1(0,t;W^{1,q})}+\|\nabla\rho\|_{L^4(0,t;L^{2})}\right)=\infty,~~~(\lambda< 2\mu)
\ee
Huang and Li \cite{[HL]}
\be\label{huangli}
\lim\sup\limits_{t\nearrow
T^\star}\left(\|\nabla
u\|_{L^1(0,t;L^\infty)}+\|\theta\|_{L^2(0,t;L^\infty)}\right)=\infty.
\ee
Sun, Wang and Zhang \cite{[SWZ1]}
\be\label{swz1}
\lim\sup\limits_{t\nearrow
T^\star}\left( \|(\rho,\f{1}{\rho},\theta)\|_{L^\infty(0,t;L^\infty)}\right)=\infty,~~~(\lambda< 7\mu).
\ee
Then we consider the case away from vacuum and state the second result as follows

\begin{theorem}
\label{th1.3}
Suppose $(\rho,u,\theta)$ is the unique strong solution in Theorem \ref{localwithout vacuum} in section 2.
If the maximal existence time $T^*$ is finite, then either of the following results holds:
 \be\label{genwithout1}\limsup_{t\nearrow T^*}\B(
  \|\rho,\ \rho^{-1} \|_{L^{\infty}(0,t;L^{\infty})}+ \|\text{div\,} u \|_{L^{p_{1}}(0,t;L^{q_{1}})}+ \|\theta\|_{L^{p_{2}}(0,t;L^{q_{2}})}\B)= \infty, \ee
  or \be\label{genwithout2}\limsup_{t\nearrow T^*}\B(
  \|\text{div\,}u \|_{L^{1}(0,t;L^{\infty})}+ \|\text{div\,} u \|_{L^{p_{1}}(0,t;L^{q_{1}})}+ \|\theta\|_{L^{p_{2}}(0,t;L^{q_{2}})}\B)= \infty, \ee
where the pairs $(p_{1},\,q_{1})$ and $(p_{2},\,q_{2})$ meet
 \be
  \f{2}{p_{1}}+\f{3}{q_{1}}=2,~ q_{1}>\f{3}{2};\,~~
  \f{2}{p_{2}}+\f{3}{q_{2}}=2,~q_{2}>\f{3}{2}.
  \ee
\end{theorem}

\begin{remark}
This theorem is an improvement of corresponding results in
\eqref{huangli} and \eqref{swz1}.
\end{remark}
\begin{remark}
Note that we donot need any additional restriction on the viscosity coefficients $\mu$ and $\lambda$.
\end{remark}

The proof of Theorem \ref{th1.3} is also enlightened by  the study of the 3D  incompressible Navier-Stokes equations. Under the natural  restriction \eqref{nares}, there holds
\be\ba\label{1.27}
\f{d}{dt}\int&\B[\f\mu2|\nabla u|^{2}+(\mu+\lambda)(\text{div\,}u  )^2+\f{1}{2\mu+\lambda}P^2 -2P\text{div\,}u +\f{C_3 C_\nu}{2}\rho\theta^2\\
	& +\f{C_4 +1}{2\mu}\rho| u|^4\B]+ \kappa\int|\nabla \theta|^{2}+\f{1}{2}\int \rho |\dot{u}|^{2}+\int
|u|^2\big|\nabla u\big|^2\\\leq& C  \int \rho^{2}|\theta|^{3}  +C \int\rho |u|^{2}|\theta|^{2}
+C\int|\text{div\,u}| |u|^{2} |\nabla u|. \ea\ee
Our observation  is that the last term in the right hand side of  \eqref{1.27} is similar to the term $\int|\Pi| |u|^{2} |\nabla u|dx$  appearing in the derivation of   regular  criteria via pressure $\Pi$ of the 3D  incompressible Navier-Stokes equations on bounded domain (see \cite{[Zhou1],[KL1],[BG],[KL2]} ).
 This  criterion
 was completely solved by Kang and   Lee \cite{[KL2]} until 2010. In the spirit of \cite{[KL2]}, we can deal with this term to derive the desired estimates.

Theorem \ref{th1.3}
immediately yields the following result.
\begin{coro}
\label{corothe1.4}
Let $(\rho,u)$ be the unique strong solution of the isentropic compressible fluid without initial vacuum.
If the maximal existence time $T^*$ is finite, then there holds
\be\label{genwithout1c}\limsup_{t\nearrow T^*}\B(
  \|\rho,\ {\rho^{-1}} \|_{L^{\infty}(0,t;L^{\infty})}+ \|\text{div\,} u \|_{L^{p }(0,t;L^{q })} \B)= \infty, \ee
  or \be\label{genwithout2c}\limsup_{t\nearrow T^*}\B(
  \|\text{div\,}u \|_{L^{1}(0,t;L^{\infty})}+ \|\text{div\,} u \|_{L^{p }(0,t;L^{q })} \B)= \infty, \ee
where the pair $(p ,\,q )$  meet
$$
  \f{2}{p }+\f{3}{q }=2,~ q >\f{3}{2}.
  $$
\end{coro}
\begin{remark}
Although this corollary is valid in the absence of vacuum, it   does not require additional assumptions on $\lambda$ and $\mu$.
A  special case of
\eqref{genwithout2c} is that
 \be \limsup_{t\nearrow T^*}
  \|\text{div\,}u \|_{L^{1}(0,t;L^{\infty})}= \infty.\ee
In the presence  of vacuum, similar blow up criteria in terms of the divergence (gradient)
of the velocity   can be found in \cite{[LX],[WZ13],[SWZ1],[HX]}.
\end{remark}

\begin{remark}
For the isentropic compressible fluid in the absence of vacuum,
combining the results proved by Huang, Li and Xin in \cite{[HLX]}
and Corollary \ref{corothe1.4}, we obtain the following  blow up criteria in terms of the  gradient  of the velocity
\be\label{genwithout1tidu}\limsup_{t\nearrow T^*}\B(
  \|\rho,\ \rho^{-1} \|_{L^{\infty}(0,t;L^{\infty})}+ \|\nabla u \|_{L^{p }(0,t;L^{q })} \B)= \infty, \ee
  or \be\label{genwithout2tidu}\limsup_{t\nearrow T^*}\B(
  \|\text{div\,}u \|_{L^{1}(0,t;L^{\infty})}+ \|\nabla u \|_{L^{p }(0,t;L^{q })} \B)= \infty, \ee
where the pair $(p ,\,q )$  meets
$$
  \f{2}{p }+\f{3}{q }=2,~q >\f{3}{2}.
  $$
  This extends   Beirao da Veiga's result in \cite{[da Veiga95]} from the incompressible Navier-Stokes equations
  to the compressible Navier-Stokes system.
\end{remark}
The remainder of this paper  is structured as follows. In section 2, we first give some notations and recall the local   strong solutions of system \eqref{FNS} due to Cho and   Kim \cite{[CK]}. We establish some auxiliary lemmas under  the hypothesis that the upper bound of the
density is bounded.
Section 3 is devoted to the proof of Theorem \ref{th1.1} and Theorem \ref{the1.2}. Section 4 contains the proof  of Theorem \ref{th1.3}.

 \section{Notations and some auxiliary lemmas} \label{sec2}
  $C$ is an absolute
   constant which may be different from line to line unless otherwise stated.
For $1\leq p\leq\infty$, $L^{p}(\mathbb{R}^{3})$ represents the usual
 Lebesgue  space. The classical Sobolev space  $W^{k,p}(\mathbb{R}^{3})$
is equipped with the norm $\|f\|_{W^{k,p}(\mathbb{R}^{3})}=\sum\limits_{\alpha =0}^{k}\|D^{\alpha}f\|_{L^{p}(\mathbb{R}^{3})}$.
A function $f$ belongs to  the homogeneous Sobolev spaces $D^{k,l}$
if $
u\in L^1_{\rm{loc}}(\mathbb{R}^3): \|\nabla^k u \|_{L^l}<\infty.$

For simplicity,   we write
$$L^p=L^p(\mathbb{R}^3),   \ H^k=W^{k,2}(\mathbb{R}^3), \ D^k=D^{k,2}(\mathbb{R}^3).$$
 We denote the $G$ by the  effective viscous flux, that is, $$G=(2\mu+\lambda)\text{div\,}u-P.$$
The notation $\dot{v}=v_t+u\cdot\nabla v$  stands for material derivative.

It is well-known that
\be \label{ntle}
\ba\|\nabla G\|_{L^{p}}\leq \|\rho \dot{u}\|_{L^{p}}, \ \forall p\in (1,+\infty).
\ea
\ee
We recall
the local well-posedness
  of   strong solutions to the full
compressible Navier-Stokes equations \eqref{FNS} due to
Cho and Kim \cite{[CK]}. The first result allows initial density contains vacuum and some compatibility
conditions are required.
The second one is  absence of vacuum. Moreover, we refer the reader to \cite{[CCK]} the local local existence and uniqueness
of strong solutions for the isentropic
compressible Navier-Stokes system.
 \begin{theorem}\label{localwith vacuum}
Suppose $ u_0, \theta_0 \in D^1(\mathbb{R}^{3})\cap D^2(\mathbb{R}^{3})$ and
\[\rho_0 \in W^{1,q}(\mathbb{R}^{3}) \cap H^1(\mathbb{R}^{3})\cap L^1(\mathbb{R}^{3})\]
for some $q\in (3,6]$.
If $\rho_0$ is nonnegative and the initial data satisfy the compatibility condition
$$\ba
&-\mu\Delta  u_0-(\mu+\lambda)\nabla\text{div}\,u_0 +\nabla P(\rho_0,\theta_0) = \sqrt{\rho_0} g_1\\
&\Delta\theta_0+\frac{\mu}{2}|\nabla  u_0 +(\nabla u_0)^{\text {tr}}|^2
+ \lambda(\text{div}\,u_0)^2 =\sqrt{\rho_0} g_2
\ea$$
for    vector fields  $g_1,g_2\in L^2(\mathbb{R}^{3})$.
Then there exist a time $T\in (0,\infty]$ and unique solution,  satisfying
\be\ba
&(\rho, u,\theta)\in C([0,T );L^1\cap H^1\cap W^{1,q}) \times C([0,T);D^1\cap D^2 )\times L^2([0,T);D^{2,q}) \\
&(\rho_t, u_t,\theta_t)\in C([0,T );L^2\cap L^q)\times L^2([0,T); D^1)\times L^2([0,T);D^1) \\
&(\rho^{1/2} u_t,\rho^{1/2}\theta_t)\in L^\infty([0,T);L^2) \times L^\infty([0,T);L^2).
\ea\ee
\end{theorem}
 \begin{theorem}\label{localwithout vacuum}
Suppose $ u_0, \theta_0 \in D^1(\mathbb{R}^{3})\cap D^2(\mathbb{R}^{3})$ and
\[\rho_0 \in W^{1,q}(\mathbb{R}^{3}) \cap H^1(\mathbb{R}^{3})\cap L^1(\mathbb{R}^{3})\]
for some $q\in (3,6]$.
If $\rho_0>0$,
then there exist a time $T\in (0,\infty]$ and unique solution,  satisfying
\be\ba
 &\rho\in C([0,T );L^1\cap H^1\cap W^{1,q}), \inf_{(x,t)\in \mathbb{R}^{3}\times[0,T]}\rho>0,   \\
&u\in C([0,T ); D^{1}\cap D^{2} )\cap L^{2}([0,T );  W^{2,q})\\&\theta\in C([0,T );  D^1 \cap D^{2} )\cap L^{2}([0,T );  W^{2,q}).
\ea\ee
\end{theorem}
Next, under the hypothesis that the upper bound of the
density is bounded, namely,
\be
  \|\rho\|_{L^{\infty}(0,T;L^{\infty})} \leq M, \label{cond1part}\ee
  we derive some useful estimates, which plays an important role in the proof of all our theorems.

\begin{lemma} Suppose that \eqref{cond1part} is valid, then there holds
\begin{equation}\label{key3}
\begin{aligned}
\f{d}{dt}\int&\B[\f\mu2|\nabla u|^{2}+(\mu+\lambda)(\text{div\,}u  )^2+\f{1}{2\mu+\lambda}P^2 -2P\text{div\,}u +\f{C_3 C_\nu}{2}\rho\theta^2 \B]
\\&\ \ \ \ \ \ + \kappa\int|\nabla \theta|^{2}+\f{1}{2}\int \rho |\dot{u}|^{2}\\
&\leq C_{4} \int \rho|\theta|^{3}  +C_{4}\int\rho |u|^{2}|\theta|^{2} +C_{4} \int |u|^{2}|\nabla u|^{2}  .
\end{aligned}
\end{equation}
\end{lemma}
\begin{proof}
Taking the $L^{2}$ inner product of the temperature equation with $\theta$, by the Cauchy inequality, we infer that
\be\ba\label{3.2}
\f{C_\nu}{2}\f{d}{dt}\int\rho\theta^{2}+\kappa\int|\nabla \theta|^{2}&\leq R\int|\rho\theta^{2}\text{div}u|+(2\mu+\lambda)\int|\nabla u|^{2}\theta \\
&\leq C\int \rho^{2}|\theta|^{3} +C\int|\nabla u|^{2}\theta .
\ea\ee
Multiplying the both sides of the momentum equation by $u\theta$ and using the
integration by parts, we get
\be\ba\label{3.3}
\mu\int|\nabla u|^{2}\theta &\leq\int|\rho \dot{u}u\theta|+|\int\nabla P u\theta dx|+C\int |u||\nabla u||\nabla \theta|\\&=I+II+III.
\ea\ee
Thanks to the Cauchy-Schwarz inequality, we find that
\be\label{3.4}
I\leq  \eta\int\rho |\dot{u}|^{2}+C(\eta)\int \rho |u|^{2}|\theta|^{2}.
\ee
According to integration by parts and Young's inequality, we conclude
\be\ba\label{3.5}
II&=|\int P \text{div}u\theta dx+\int P  u\nabla\theta dx|\\
&=|R\int \rho \text{div}u\theta^{2} +R\int \rho \theta u\nabla\theta  |\\
&\leq  C\int \rho^{2}|\theta|^{3} +\f{\mu}{8}\int|\nabla u|^{2}\theta + \varepsilon_1 \int|\nabla\theta|^{2} +C\int \rho^{2}\theta^{2}|u|^{2}  .
\ea\ee
The Cauchy-Schwarz inequality yields that
\be \label{3.6}
III\leq \varepsilon_1\int  |\nabla \theta|^{2}+C\int |u|^{2}|\nabla u|^{2}  .
\ee
Plugging \eqref{3.4}-\eqref{3.6} into  \eqref{3.3}, we have
\be\ba\label{3.7}\f{7\mu}{8}\int|\nabla u|^{2}\theta \leq&  \eta\int|\rho \dot{u}|^{2}+C\int\rho |u|^{2}|\theta|^{2}+C\int \rho^{2}|\theta|^{3}  +C\int \rho^{2}\theta^{2}|u|^{2}  \\&+2\varepsilon_1\int  |\nabla \theta|^{2}+C\int |u|^{2}|\nabla u|^{2}  .
\ea
\ee
It follows from \eqref{3.2} and \eqref{3.7} that
\be\ba\label{k1}
&\f{C_\nu}{2}\f{d}{dt}\int\rho\theta^{2}+ \kappa \int|\nabla \theta|^{2}\\&\leq C_{1}\int \rho^{2}|\theta|^{3} + C_{1} \eta\int|\rho \dot{u}|^{2}+C_{1}\int\rho |u|^{2}|\theta|^{2} +C_{1} \int |u|^{2}|\nabla u|^{2}  .
\ea
\ee
Taking the $L^{2}$ inner product with $u_{t}$ in the second equation of \eqref{FNS}, we get
\be\ba\label{3.10}
&\f{1}{2}\f{d}{dt}\int\B[\mu|\nabla u|^{2}+(\lambda+\mu)(\text{div\,}u)^{2}\B]+\int \rho |\dot{u}|^{2}\\=&\int \rho \dot{u}(u\cdot\nabla u)+\int P\text{div\,}u_{t}dx\\=&J_1+J_2.
\ea
 \ee
The Young inequality ensures that
 \be\ba\label{3.11}
J_1\leq\f{1}{4}\int \rho |\dot{u}|^{2}+C\int |u|^{2}|\nabla u|^{2}  .
\ea
 \ee
 After a few calculations, by the effective viscous flux $G=(2\mu+\lambda)\text{div\,}u-P$,  we arrive at
 \be\ba\label{3.12}
J_2&=\f{d}{dt}\int P\text{div\,}u-\int P_{t}\text{div\,}u\\&=\frac{d}{dt}\int P\text{div\,}u-\frac{1}{2(2\mu+\lambda)}
\frac{d}{dt}\int P^2
-\frac{1}{2\mu+\lambda}\int P_tG
\\&=J_{21}+J_{22}+J_{23}.
\ea
 \ee
 Notice that the equation of $\rho E=P+\frac{\rho|u|^2}{2}$  is governed by
 \be\label{eqofE}
 (\rho E)_t+\text{div\,}(\rho E u+P
u)-\kappa\Delta \theta =\text{div\,}\B\{\big[\lambda\text{div\,}u
Id+\mu (\nabla u+(\nabla u)^{\text{tr}})\big]\cdot u\B\}.
\ee
By virtue of \eqref{eqofE},
we see that
\be  \ba\label{3.13}
J_{23}=&-\frac{1}{2\mu+\lambda}\int\ (\rho E)_t G+
\frac{1}{2\mu+\lambda}\int\ \left(\frac{\rho|u|^2}{2}\right)_t G\\=&-\frac{R}{2\mu+\lambda}\int\ \rho \theta u\cdot\nabla G-\frac{1}{2\mu+\lambda}\int\ \rho\frac{|u|^2}{2}u\cdot\nabla G\\&+\frac{1}{2\mu+\lambda}\int\ \B\{\big[\lambda\text{div\,}u
Id+\mu (\nabla u+(\nabla u)^\prime)\big]\cdot u\B\}\nabla G+\frac{\kappa}{2\mu+\lambda}\int\ \nabla\theta\cdot\nabla G.
\\&+
\frac{1}{2\mu+\lambda}\int\ \left(\frac{\rho|u|^2}{2}\right)_t G. \ea
\ee
With the help of the Young inequality,  \eqref{ntle} and \eqref{cond1part}, we get
$$
-\frac{1}{2\mu+\lambda}\int\ \rho \theta u\cdot\nabla G\leq \f{\eta}{4}\|\nabla G\|^{2}_{L^2}+C\int \rho^{2} |u|^{2}|\theta|^{2}\leq \varepsilon \|\rho \dot{u}\|^{2}_{L^2} +C\int \rho  |u|^{2}|\theta|^{2}
$$
Likewise, there hold
$$\ba
\frac{1}{2\mu+\lambda}\int\ \B[\lambda\text{div\,}u
Id+\mu (\nabla u+(\nabla u)^\prime)\B]
u\nabla G &\leq \f{\eta}{4}\|\nabla G\|^{2}_{L^2}+C\int |u|^{2}|\nabla u|^{2}  \\
& \leq \varepsilon\|\rho \dot{u}\|^{2}_{L^2}+C\int |u|^{2}|\nabla u|^{2}  , \\
\ea$$
$$\ba
&\frac{1}{2\mu+\lambda}\int\ \nabla\theta\cdot\nabla G \leq \f{\eta}{4}\|\nabla G\|^{2}_{L^2}+C \int|\nabla \theta|^{2}dx\leq \varepsilon\|\rho \dot{u}\|^{2}_{L^2}+C \int|\nabla \theta|^{2}dx.
\ea$$
Putting together with the above estimates, we have
\be \ba\label{3.14}
-\frac{1}{2\mu+\lambda}\int\ (\rho E)_t G
\leq & -\frac{1}{2\mu+\lambda}\int\ \rho\frac{|u|^2}{2}u\cdot\nabla G+ \varepsilon \int \rho |\dot{u}|^{2}dx+C\int \rho  |u|^{2}|\theta|^{2}dx\\&
+C\int |u|^{2}|\nabla u|^{2}+C \int|\nabla \theta|^{2}.
\ea\ee
We turn our attentions to the last term of \eqref{3.13}. A straightforward calculation gives
\be \ba\label{3.15}
\frac{1}{2\mu+\lambda}\int\ \left(\frac{\rho|u|^2}{2}\right)_t G=&\frac{1}{2\mu+\lambda}\int\ \frac{\rho_t|u|^2}{2} G+
\frac{1}{2\mu+\lambda}\int\ \rho u\cdot u_t G.
\ea
\ee
Taking the advantage of $ \rho_t=-\text{div\,}(\rho u)$, the integration by parts, the Young inequality and \eqref{cond1}, we get
\be\ba\label{3.16}
\frac{1}{2\mu+\lambda}&\int\ \frac{\rho_t|u|^2}{2} G=-\frac{1}{2\mu+\lambda}\int\ \frac{\text{div\,}(\rho u)|u|^2}{2} G\\=&\frac{1}{2\mu+\lambda}\int\ \rho u\cdot\nabla u\cdot u G+\frac{1}{2\mu+\lambda}\int\ \frac{\rho u|u|^2}{2}\cdot\nabla G\\
\leq& C\int\ \rho |u\cdot\nabla u|^{2}+C\int\ \rho |u|^{2}| G|^{2}+\frac{1}{2\mu+\lambda}\int\ \frac{\rho u|u|^2}{2}\cdot\nabla G\\
\leq& C\int\ \rho |u\cdot\nabla u|^{2}+C\int\ \rho |u|^{2}(|\nabla u|^{2}+R\rho^{2}\theta^{2})+\frac{1}{2\mu+\lambda}\int\ \frac{\rho u|u|^2}{2}\cdot\nabla G\\
\leq& C\int\   |u|^{2} |\nabla u|^{2}+C\int\ \rho |u|^{2} \theta^{2} +\frac{1}{2\mu+\lambda}\int\ \frac{\rho u|u|^2}{2}\cdot\nabla G,
\ea\ee
where we have used the fact
\be\label{nxtlc}|G|\leq C(|\nabla u|+|\rho\theta|).\ee
From $\dot{u}=u_t+u\cdot\nabla u$, the Young inequality, \eqref{nxtlc} and \eqref{cond1}, we find
\be\ba\label{3.18}
\frac{1}{2\mu+\lambda}\int\ \rho u\cdot u_t G
&=\frac{1}{2\mu+\lambda}\int\ \rho u\cdot (\dot{u}-u\cdot\nabla u) G\\&=\frac{1}{2\mu+\lambda}\int\ \rho u\cdot \dot{u}G-\int\ \rho u u\cdot\nabla u  G\\
&\leq \varepsilon\int \rho |\dot{u}|^{2}dx+C\int \rho |  u |^{2}|G |^{2}dx +\int\ \rho |u|^2 |\nabla u|^2
\\
&\leq \varepsilon\int \rho |\dot{u}|^{2}dx+C\int\   |u|^{2} |\nabla u|^{2}+C\int\ \rho |u|^{2} \theta^{2}.
 \ea\ee
Inserting  \eqref{3.16} and \eqref{3.18} into \eqref{3.15}, we obtain
 \be  \ba\label{3.19}
 \frac{1}{2\mu+\lambda}\int\ \left(\frac{\rho|u|^2}{2}\right)_t G\leq& \varepsilon\int \rho |\dot{u}|^{2}+C\int\   |u|^{2} |\nabla u|^{2}\\&+C\int\ \rho |u|^{2} \theta^{2}+ \frac{1}{2\mu+\lambda}\int\ \frac{\rho u|u|^2}{2}\cdot\nabla G.
\ea\ee
We derive from \eqref{3.13}, \eqref{3.14} and \eqref{3.19} that
 \be  \ba\label{3.20}
&J_{23}
\leq &\varepsilon\int \rho |\dot{u}|^{2}+C\int \rho  |u|^{2}|\theta|^{2} +C\int |u|^{2}|\nabla u|^{2} +C \int|\nabla \theta|^{2}.
\ea\ee
It follows from \eqref{3.10}, \eqref{3.11}, \eqref{3.12} and \eqref{3.20} that
\be\ba\label{k2}
&  \f{1}{2} \f{d}{dt}\int\B[\mu|\nabla u|^{2}+(\lambda+\mu)(\text{div\,}u)^{2}+ \frac{1}{ (2\mu+\lambda)} P^2-2P\text{div\,}u\B]+ \f{1}{2} \int \rho |\dot{u}|^{2}\\
&\leq C_{2}\int \rho  |u|^{2}|\theta|^{2} +C_{2}\int |u|^{2}|\nabla u|^{2} +C_{2} \kappa\int|\nabla \theta|^{2}.
\ea\ee
Since $P=R\rho\theta$,
we can choose $C_{3}\geq C_{2}+1$ and $C_{3}$ sufficiently large to make sure that
$$
\f{\mu}{2}|\nabla u|^{2}+(\mu+\lambda)(\text{div\,}u)^2+\f{1}{2\mu+\lambda}P^2-2P\text{div\,}u+\f{C_3C_\nu}{2}\rho\theta^2\geq
\f\mu2|\nabla u|^{2}+\rho\theta^{2}.
$$
Multiplying\eqref{k1} both sides by $C_{3}$ and adding it with  \eqref{k2}, we end up with
\be\ba
& \f{d}{dt}\int\B[\f\mu2|\nabla u|^{2}+(\mu+\lambda)(\text{div\,}u  )^2+\f{1}{2\mu+\lambda}P^2 -2P\text{div\,}u +\f{C_3 C_\nu}{2}\rho\theta^2 \B]\\
&\ \ +C_{3}\kappa\int|\nabla \theta|^{2}+\int \rho |\dot{u}|^{2}\\&\leq C_{1}C_{3}\int \rho^{2}|\theta|^{3} +C_{1}C_{3}  \varepsilon\int|\rho \dot{u}|^{2}+C_{1}C_{3}\int\rho |u|^{2}|\theta|^{2} +C_{1}C_{3} \int |u|^{2}|\nabla u|^{2}  \\&\ \ + C_{2}\int \rho  |u|^{2}|\theta|^{2} +C_{2}\int |u|^{2}|\nabla u|^{2} +C_{2} \kappa\int|\nabla \theta|^{2}\\
&\leq C\int \rho^2\theta^3+C_1C_3\varepsilon\int \rho\dot{u}^2+C\int \rho |u|^2|\theta|^2+C\int |u|^2|\nabla{u}|^2+C_2\kappa\int |\nabla{\theta}|^2.
\ea\ee
Choosing $\varepsilon $ sufficiently small to obtain \eqref{key3}.
This completes the proof of this lemma.
\end{proof}

\section{Blow up criteria with vacuum}

\subsection{Extra constraint on the coefficients of viscosity}

In what follows, we assume that $(\rho, u, \theta)$ is a strong solution of (1) in $ [0, T )\times \mathbb{R}^{3}$
with the regularity stated in Theorem \ref{localwith vacuum}. We will prove Theorem \ref{th1.1} by a contradiction
argument. Therefore, we   assume that
\be
  \|\rho\|_{L^{\infty}(0,T;L^{\infty})}+\|\theta\|_{L^{p}(0,T;L^{q})}\leq C,\ \  \f{2}{p}+\f{3}{q}=2,\ \text{where}\ q> \f{3}{2} .\label{cond1}\ee
First, we follow the arguments of Wen and Zhu \cite{[WZ13]} and
Li, Xu and Zhu \cite{[LXZ]} to prove the lemma below.

\begin{lemma}\label{lema3.2} Suppose that \eqref{cond1} is valid and $\lambda<3\mu$, then  there holds
\be \label{3.23}
\frac{d}{dt}\int\ \rho|u|^4+ \int_{{\mathbb{R}^3}\cap\{|u|>0\}}
|u|^2\big|\nabla u\big|^2\leq C_{5}\int\rho |u|^{2}|\theta|^{2} .\ee
\end{lemma}\begin{proof}
Multiplying the momentum equations by $4|u|^{2}u$ and integrating on $\mathbb{R}^{3}$, we find
\be  \ba\label{3.24}
&\frac{d}{dt}\int\
\rho|u|^4+\int_{{\mathbb{R}^3}\cap\{|u|>0\}} 4|u|^{2}\B[\mu|\nabla
u|^2+(\lambda+\mu)|\text{div\,}u|^2+2\mu|\nabla|u||^2\B]\\
=&4\int_{{\mathbb{R}^3}\cap\{|u|>0\}}
\text{div\,}(|u|^{2}u)P-8(\mu+\lambda)\int_{{\mathbb{R}^3}\cap\{|u|>0\}}\text{div\,}u|u|u\cdot\nabla|u|
\\=& K_1+K_2.
\ea\ee
Using the Young inequality twice, for $\varepsilon_{0}\in(0,\f14)$, we get
\be\ba\label{3.25}
K_1&=4R\int_{{\mathbb{R}^3}\cap\{|u|>0\}}
(|u|^{2}\mathrm{div\,}u+2u\cdot\nabla u\cdot u)\rho\theta
\\&\leq 2\mu \varepsilon_{0} \int |u|^{2}|\nabla u|^{2}+C\int\rho |u|^{2}|\theta|^{2} +2\mu \varepsilon_{0} \int |u|^{2}|\nabla u|^{2}+C\int\rho |u|^{2}|\theta|^{2}
\\&\leq  4\mu  \varepsilon_{0} \int |u|^{2}|\nabla u|^{2}+C\int\rho |u|^{2}|\theta|^{2}.
\ea\ee
By the Cauchy inequality, we have
\be\ba\label{3.27}
&K_2&\leq 4 (\lambda+\mu)\int_{{\mathbb{R}^3}\cap\{|u|>0\}} |u|^{2}|\text{div\,}u|^2 +4(\mu+\lambda) \int_{{\mathbb{R}^3}\cap\{|u|>0\}}|u|^{2}\big|\nabla|u|\big|^2.
\ea\ee
Substituting \eqref{3.25} and
\eqref{3.27} into \eqref{3.24}, we conclude that
 \be
\ba\frac{d}{dt}\int\
\rho|u|^4&+4\mu \int_{{\mathbb{R}^3}\cap\{|u|>0\}} |u|^{2}|\nabla u|^2+8\mu \int_{{\mathbb{R}^3}\cap\{|u|>0\}}|u|^{2}\big|\nabla|u|\big|^2
\\
\le& 4\mu  \varepsilon_{0} \int |u|^{2}|\nabla u|^{2}+C\int\rho |u|^{2}|\theta|^{2} + 4(\mu+\lambda) \int_{{\mathbb{R}^3}\cap\{|u|>0\}}|u|^{2}\big|\nabla|u|\big|^2.
\ea\ee
By means of $|\nabla
u|^2=|u|^2\left|\nabla\left(\frac{u}{|u|}\right)
\right|^2+\big|\nabla|u|\big|^2, $
we further derive that
$$
\ba\frac{d}{dt}&\int\
\rho|u|^4+4\mu (1-\varepsilon_{0})\int_{{\mathbb{R}^3}\cap\{|u|>0\}}|u|^{2}\B(\big|\nabla
|u|\big|^2+ |u|^2\left|\nabla
\left(\frac{u}{|u|}\right)\right|^2\B)
\\&\ \ \ \ \ \ \ \ \ \ \ \ +8\mu\int_{{\mathbb{R}^3}\cap\{|u|>0\}}|u|^{2}\big|\nabla|u|\big|^2
\\
\le&  C\int\rho |u|^{2}|\theta|^{2} + 4(\mu+\lambda) \int_{{\mathbb{R}^3}\cap\{|u|>0\}}|u|^{2}\big|\nabla|u|\big|^2\\
&\ \ \ \ \ \ \ \ \ \ \ \ \ .
\ea$$
that is,
\be\label{c16}
\ba\frac{d}{dt}&\int\
\rho|u|^4+4\int_{{\mathbb{R}^3}\cap\{|u|>0\}}
\mu(1-\varepsilon_0)|u|^4\big|\nabla(\f{u}{|u|})\big|^2+\left(\mu(2-\varepsilon_0)-\lambda\right)|u|^2\big|\nabla|u|\big|^2\\
&\leq  C\int\rho |u|^{2}|\theta|^{2},
 \ea\ee
 if $\lambda<(2-\varepsilon_0)\mu$, then we have already proved the Lemma \ref{lema3.2}. To deal with the case of $\lambda\geq(2-\varepsilon_0)\mu$, we define the function introduced by \cite{[LXZ]} below
$$
\phi(\varepsilon_0)=\left\{\begin{array}{l}
\frac{\mu }{ (\lambda+\mu\varepsilon_0)}, \ \ \ {\rm if} \ \ \ \lambda+\mu\varepsilon_0>0, \\
[3mm] 0, \ \ \ \ \ \ \ \ \ \ \ \ \ \ \ \ \ {\rm otherwise}.
\end{array}
\right.
$$
{\bf Case 1:} Assume that
\be
\label{non-blow-up:5.8}\int_{{\mathbb{R}^3}\cap\{|u|>0\}}
|u|^4\left|\nabla\left(\frac{u}{|u|}\right)\right|^2>
\phi(\varepsilon_0)\int_{{\mathbb{R}^3}\cap\{|u|>0\}}
|u|^{2}\big|\nabla|u|\big|^2.
\ee
Combining the \eqref{c16} and \eqref{non-blow-up:5.8}, we have
\be
\ba\frac{d}{dt}&\int\
\rho|u|^4+4f(\varepsilon_0)\int_{{\mathbb{R}^3}\cap\{|u|>0\}}
|u|^2\big|\nabla|u|\big|^2\leq  C\int\rho |u|^{2}|\theta|^{2},
 \ea\ee
 where
 \be\label{a2}\ba f(\varepsilon_0)&=\mu(1-\varepsilon_0)\phi(\varepsilon_0)+\left(\mu(2-\varepsilon_0)-\lambda\right)=\f{\mu^2(1-\varepsilon_0)}{\lambda+\mu\varepsilon_0}+(2-\varepsilon_0)\mu-\lambda\\
 &=-\f{\left(\lambda-a_1(\varepsilon_0)\mu\right)\left(\lambda-a_2(\varepsilon_0)\mu\right)}{\lambda+\mu\varepsilon_0},
 \ea\ee
 in which \be\ba
 a_1(\varepsilon_0)=1-\varepsilon_0+\sqrt{2-\varepsilon_0},\ \ a_2(\varepsilon_0)=1-\varepsilon_0-\sqrt{2-\varepsilon_0}.
 \ea\ee
 Since $\lambda-a_2(\varepsilon_0)\mu=\mu \varepsilon_0+\lambda +(\sqrt{2-\varepsilon_0}-1)\mu>0$ when $\lambda+\mu\varepsilon_0>0$,
then by the condition $(2-\varepsilon_0)\mu\leq\lambda <3\mu$, we can easily check that $f(\varepsilon_0)>0$.

Hence, there holds
$$
\frac{d}{dt}\int\
\rho|u|^4+ \int_{{\mathbb{R}^3}\cap\{|u|>0\}} |u|^{2} |\nabla
u|^2\leq C\int\rho |u|^{2}|\theta|^{2}
.$$
{\bf Case 2:} if
\be\label{blow-up:5.3}
\int_{{\mathbb{R}^3}\cap\{|u|>0\}}|u|^4
\left|\nabla\left(\frac{u}{|u|}\right)
\right|^2\le\phi(\varepsilon_0)
\int_{{\mathbb{R}^3}\cap\{|u|>0\}}
|u|^{2}\big|\nabla|u|\big|^2.\ee
 A direct calculation gives for $|u|>0$ \be\label{blow-up:5.5}\begin{split}
&\text{div\,}u=|u|\text{div\,}\left(\frac{u}{|u|}\right)+\frac{u\cdot\nabla|u|}{|u|}.
\end{split}\ee
Plugging this into the last term of \eqref{3.24}, we write
 \be
\begin{split}&\frac{d}{dt}\int\  \rho|u|^4+4\int_{{\mathbb{R}^3}\cap\{|u|>0\}} |u|^{2}\left(\mu|\nabla u|^2+(\lambda+\mu)|\text{div\,}u|^2+2\mu \big|\nabla|u|\big|^2\right)\\
=& K_1-8(\mu+\lambda)\int_{{\mathbb{R}^3}\cap\{|u|>0\}}|u|^{2}\text{div\,}\left(\frac{u}{|u|}\right)u\cdot\nabla|u|
-8(\mu+\lambda)\int_{{\mathbb{R}^3}\cap\{|u|>0\}} \big|u\cdot\nabla|u|\big|^2,
\end{split}\ee
namely, \be \label{dt rho u r}
\begin{split}&\frac{d}{dt}\int\  \rho|u|^4+4\int_{{\mathbb{R}^3}\cap\{|u|>0\}}  H
= K_1,
\end{split}\ee  where
\be \begin{split} H=&\mu|u|^2|\nabla
u|^2+(\lambda+\mu)|u|^2|\text{div\,}u|^2+2\mu |u|^2\big|\nabla|u|\big|^2
\\&+2(\mu+\lambda)|u|^2\text{div\,}\left(\frac{u}{|u|}\right)u\cdot\nabla|u|
+2(\mu+\lambda)\big|u\cdot\nabla|u|\big|^2.\end{split}
\ee
Combining this and \eqref{3.25}, we deduce
\be \frac{d}{dt}\int\  \rho|u|^4+4\int_{{\mathbb{R}^3}\cap\{|u|>0\}}  G
\leq C_{6}\int\rho |u|^{2}|\theta|^{2},
 \ee  where
\be \begin{split} G=&\mu(1-\varepsilon_{0})|u|^2|\nabla
u|^2+(\lambda+\mu)|u|^2|\text{div\,}u|^2+2\mu|u|^2\big|\nabla|u|\big|^2
\\&+2(\mu+\lambda)|u|^2\text{div\,}\left(\frac{u}{|u|}\right) u\cdot\nabla|u|+2(\mu+\lambda)\big|u\cdot\nabla|u|\big|^2.
\end{split}
\ee
Now we analyze the positiveness of the term $G$ as follows
\be \label{c17}\begin{split}
G&=\mu(1-\varepsilon_{0})|u|^4|\nabla(\f{u}{|u|})
|^2+\mu(3-\varepsilon_{0})|u|^2\big|\nabla|u|\big|^2\\
&\ \ +3(\mu+\lambda)\left(u\cdot\nabla |u|+\f{2}{3}|u|^2\text{div\,}(\f{u}{|u|})\right)^2-\f{1}{3}(\mu+\lambda)|u|^4\big|\text{div\,}(\f{u}{|u|})\big|^2\\
&\geq \mu(1-\varepsilon_{0})|u|^4\big|\nabla(\f{u}{|u|})\big|^2+\mu(3-\varepsilon_{0})|u|^2\big|\nabla|u|\big|^2-(\mu+\lambda)
|u|^4\big|\nabla(\f{u}{|u|})\big|^2\\
&\geq -(\lambda+\varepsilon_0\mu)|u|^4\big|\nabla(\f{u}{|u|})\big|^2+\mu(3-\varepsilon_{0})|u|^2\big|\nabla|u|\big|^2\\
&\geq \left(\mu(3-\varepsilon_0)-\phi(\varepsilon_0)(\lambda+\varepsilon_0\mu)\right)|u|^2\big|\nabla|u|\big|^2\\
&\geq \mu(2-\varepsilon_0)|u|^2\big|\nabla|u|\big|^2,
\end{split}
\ee
where we have used the fact $|\text{div\,}(\f{u}{|u|})|^2\leq 3|\nabla (\f{u}{|u|})|^2$.
With this in hand, we know that
\be \label{key4}
\frac{d}{dt}\int\ \rho|u|^4+ \int_{{\mathbb{R}^3}\cap\{|u|>0\}}
|u|^2\big|\nabla u\big|^2\leq C\int\rho |u|^{2}|\theta|^{2} ,\ee
Summary, no matter in which case, we always have \eqref{3.23}. This proves Lemma \ref{lema3.2}.
\end{proof}

\begin{lemma}\label{lemma3.3} Suppose that \eqref{cond1} is valid and $\lambda<3\mu$, then there holds
\be \label{3.39}
\sup_{0\leq t\leq T}\int\B[\f\mu2|\nabla u|^{2}+\rho\theta^{2}
+\rho|u|^4\B] + \kappa\int_{0}^{T}\int|\nabla \theta|^{2} + \int_0^T\int \rho |\dot{u}|^{2} +
|u|^2\big|\nabla u\big|^2\leq C.\ee
\end{lemma}
\begin{proof}
Multiplying the inequality  {\eqref{key4}} by $(C_{4}+1)$ and adding the result to the inequality \eqref{key3}, we can obtain
\be\ba\label{3.40}
\f{d}{dt}\int&\B[\f\mu2|\nabla u|^{2} +\rho\theta^{2}
+\rho|u|^4\B]+ \kappa\int|\nabla \theta|^{2}+\f{1}{2}\int \rho |\dot{u}|^{2}\\
&+\int_{{\mathbb{R}^3}\cap\{|u|>0\}}
|u|^2\big|\nabla u\big|^2\leq C_{7} \left(\int \rho|\theta|^{3}  +\int\rho |u|^{2}|\theta|^{2}\right).
 \ea\ee
At this stage, it suffices to bound the right hand side of
\eqref{3.40}.
Indeed, by the interpolation inequality, \eqref{cond1} and the Young inequality imply  that
\be\ba\label{3.42}
\int\rho |u|^{2}|\theta|^{2}&=\int\rho^{\f12}|\theta|\rho^{\f12}|u|^{2}|\theta|   \\
&\leq \|\theta\|_{L^{q}}\|\rho^{\f12}\theta\|_{L^{\f{2q}{q-1}}}
\|\rho^{\f12}|u|^{2}\|_{L^{\f{2q}{q-1}}}\\
&\leq \|\theta\|_{L^{q}}\|\rho^{\f12}\theta\|^{1-\f{3}{2q}}_{L^{2}}\|
\rho^{\f12}\theta\|^{\f{3}{2q}}_{L^{6}}\|\rho^{\f12}|u|^{2}\|_{L^{2}}^{1-\f{3}{2q}}\|
\rho^{\f12}|u|^{2}\|^{\f{3}{2q}}_{L^{6}}\\
&\leq  C \|\theta\|_{L^{q}}\|\rho^{\f12}\theta\|^{1-\f{3}{2q}}_{L^{2}}\|\rho^{\f12}|u|^{2}\|_{L^{2}}^{1-\f{3}{2q}}\|
 (\|
 \nabla\theta\|^{\f{3}{q}}_{L^{2}}+\|\nabla|u|^{2}\|^{\f{3}{q}}_{L^{2}})
 \\
&\leq  C \|\theta\|^{\f{2q}{2q-3}}_{L^{q}}\|\rho^{\f12}\theta\|_{L^{2}}\|\rho^{\f12}|u|^{2}\|_{L^{2}}
 +\eta_{1}
 \|\nabla\theta\|^{2}_{L^{2}}+\eta_{2}\|\nabla|u|^{2}\|^{2}_{L^{2}} \\
&\leq  C \|\theta\|^{\f{2q}{2q-3}}_{L^{q}}\big(\|\rho^{\f12}\theta\|_{L^{2}}^{2}+\|\rho^{\f12}|u|^{2}\|_{L^{2}}^{2}\big)
 +\eta_{1}
 \|\nabla\theta\|^{2}_{L^{2}}+\eta_{2}\|\nabla|u|^{2}\|^{2}_{L^{2}}  \ea
\ee
By similar above arguments, we can get
\be\ba \label{3.43}
\int\rho |\theta|^{3}&=\int\rho^{\f12}|\theta|\rho^{\f12}\theta|\theta|   \\
&\leq \|\theta\|_{L^{q}}\|\rho^{\f12}\theta\|_{L^{\f{2q}{q-1}}}^{2}\\
&\leq \|\theta\|_{L^{q}}\|\rho^{\f12}\theta\|^{2-\f{3}{q}}_{L^{2}}\|
\rho^{\f12}\theta\|^{\f{3}{q}}_{L^{6}}\\
&\leq  C \|\theta\|_{L^{q}}\|\rho^{\f12}\theta\|^{2-\f{3}{q}}_{L^{2}}\|
 \theta\|^{\f{3}{q}}_{L^{6}}\\
&\leq  C \|\theta\|_{L^{q}}\|\rho^{\f12}\theta\|^{2-\f{3}{q}}_{L^{2}}\|
 \nabla\theta\|^{\f{3}{q}}_{L^{2}}\\
&\leq  C \|\theta\|^{\f{2q}{2q-3}}_{L^{q}}\|\rho^{\f12}\theta\|_{L^{2}}^{2}
 +\eta_{3}\| \nabla\theta\|^{2}_{L^{2}}
 \ea
\ee
Substituting  \eqref{3.42} and \eqref{3.43} into \eqref{3.40}, we have
\be\ba\label{3.44}
&\f{d}{dt}\int\B[\f\mu2|\nabla u|^{2}+(\mu+\lambda)(\text{div\,}u  )^2+\f{1}{2\mu+\lambda}P^2 -2P\text{div\,}u +\f{C_3 C_\nu}{2}\rho\theta^2 +\f{C_4 +1}{2\mu}\rho| u|^4\B]\\
&\ \ + \f\kappa2\int|\nabla \theta|^{2}+\f{1}{2}\int \rho |\dot{u}|^{2}+\f{1}{2}\int_{{\mathbb{R}^3}\cap\{|u|>0\}}
|u|^2\big|\nabla u\big|^2\\
&\leq C \|\theta\|^{\f{2q}{2q-3}}_{L^{q}}\big(\|\rho^{\f12}\theta\|_{L^{2}}^{2}+\|\rho^{\f12}|u|^{2}\|_{L^{2}}^{2}\big)\\
&\leq C\|\theta\|_{L^q}^{\f{2q}{2q-3}}\B(\int\B[\f\mu2|\nabla u|^{2}+(\mu+\lambda)(\text{div\,}u  )^2+\f{1}{2\mu+\lambda}P^2 -2P\text{div\,}u\\
&\ \ \ \ \ \ \ \ \ \ \ \ \ \ \ \ \ \ \ \  +\f{C_3 C_\nu}{2}\rho\theta^2 +\f{C_4 +1}{2\mu}\rho| u|^4\B]\B),
 \ea\ee
 where we used the fact that
 $$\int\B[ (\mu+\lambda)(\text{div\,}u)^2+\f{1}{2\mu+\lambda}P^2-2P\text{div\,}u +\f{C_3C_\nu}{2}\rho\theta^2\B]\geq \int \rho\theta^2,$$
 provided that the constant $C_3$ is suitable large enough.

Then, The Gronwall lemma and \eqref{3.44} enables us to obtain that
 \be\sup\limits_{0\le t\le T}\int\ (\rho\theta^{2}+|\nabla u|^2+\rho |u|^4) +\int_0^T\int\ \rho |\dot{u}|^2+|\nabla\theta|^{2}+|u|^2\big|\nabla u\big|^2\,\le C.\ee
\end{proof}

\begin{proof}[Proof of Theorem \ref{th1.1}]

With Lemma  \ref{lemma3.3} at our disposal, according to \eqref{cond1} and \eqref{HL2} (alternatively, \eqref{wzsiam}), we completes the proof of this theorem.
\end{proof}

 \subsection{Without extra constraint on   the coefficients of viscosity}

As mentioned in the last subsection, it suffices to prove Lemma \ref{lemma3.3} without $\lambda<3\mu$ to show Theorem \ref{the1.2}.

\begin{proof}[Proof of Lemma \ref{lemma3.3} without $\lambda<3\mu$]
As the Lemma \ref{lema3.2}, there holds
\be  \ba\label{34.1}
&\frac{d}{dt}\int\
\rho|u|^4+\int 4|u|^{2}\B[\mu|\nabla
u|^2+(\lambda+\mu)|\text{div\,}u|^2+2\mu|\nabla|u||^2\B]\\
=&4\int
\text{div\,}(|u|^{2}u)P-8(\mu+\lambda)\int \text{div\,}u|u|u\cdot\nabla|u|
\\=& L_1+L_2.
\ea\ee
Making use of the Young inequality twice, we have
\be\ba\label{34.2}
L_1&=4R\int
(|u|^{2}\mathrm{div\,}u+2u\cdot\nabla u\cdot u)\rho\theta
\\&\leq \eta_{1}\int |u|^{2}|\nabla u|^{2}+C\int\rho |u|^{2}|\theta|^{2} +\eta_{2}\int |u|^{2}|\nabla u|^{2}+C\int\rho |u|^{2}|\theta|^{2}
\\&\leq  (\eta_{1}+\eta_{2})\int |u|^{2}|\nabla u|^{2}+C\int\rho |u|^{2}|\theta|^{2}.
\ea\ee
Similarly,
\be\ba\label{34.3}
&L_2&\leq C(\eta)\int |u|^{2}|\text{div\,}u|^2 +\eta\int|u|^{2}\big|\nabla|u|\big|^2.
\ea\ee
Plugging \eqref{34.2} and \eqref{34.3} into \eqref{34.1},
we get
 \be  \ba  \label{34.4}
&\frac{d}{dt}\int\
\rho|u|^4+2\mu\int|u|^{2}|\nabla
u|^2\leq C \int\rho |u|^{2}|\theta|^{2}
 +C\int |u|^{2}|\text{div\,}u|^2.
  \ea\ee
  Adding  $\eqref{34.4}$ multiplied by  $ \f{C_{4}+1}{2\mu}$ to $  \eqref{key3}$, we have
\be\ba\label{34.5}
&\f{d}{dt}\int\B[\f\mu2|\nabla u|^{2}+(\mu+\lambda)(\text{div\,}u  )^2+\f{1}{2\mu+\lambda}P^2 -2P\text{div\,}u +\f{C_3 C_\nu}{2}\rho\theta^2 \\
&\ \ \ \ +\f{C_4 +1}{2\mu}\rho| u|^4\B]+ \kappa\int|\nabla \theta|^{2}+\f{1}{2}\int \rho |\dot{u}|^{2}+\mu\int|u|^{2}|\nabla
u|^2 \\
&\leq C  \int \rho|\theta|^{3}   +C \int\rho |u|^{2}|\theta|^{2} + C \int |u|^{2}|\text{div\,}u|^2.
\ea\ee
 Case 1:
 \be
  \|\rho\|_{L^{\infty}(0,T;L^{\infty})}+\|\text{div\,}u\|_{L^{2}(0,T;L^{3})}
  +\|\theta\|_{L^{p}(0,T;L^{q})}\leq C.\label{cond2}\ee
With the help of H\"older inequality and Sobolev inequality, we get
\be\ba\label{34.6}
 \int |u|^{2}|\text{div\,}u|^2
 &\leq C\B(\int |u|^{6}\B) ^{\f13}\B(\int  |\text{div\,}u|^3 \B) ^{\f23}\\
  &\leq C\B(\int |\nabla u|^{2}\B) \B(\int  |\text{div\,}u|^3 \B) ^{\f23}.
 \ea\ee
Inserting  \eqref{3.42}, \eqref{3.43} and \eqref{34.6} into \eqref{34.5}, we conclude that
  \be\ba\label{34.7}
&\f{d}{dt}\int\B[\f\mu2|\nabla u|^{2}+(\mu+\lambda)(\text{div\,}u  )^2+\f{1}{2\mu+\lambda}P^2 -2P\text{div\,}u +\f{C_3 C_\nu}{2}\rho\theta^2\\
&\ \  +\f{C_4 +1}{2\mu}\rho| u|^4\B]+ \kappa\int|\nabla \theta|^{2}+\f{1}{2}\int \rho |\dot{u}|^{2}+\mu\int|u|^{2}|\nabla
u|^2\\&\leq C \|\theta\|^{\f{2q}{2q-3}}_{L^{q}}\big(\|\rho^{\f12}\theta\|_{L^{2}}^{2}+\|\rho^{\f12}|u|^{2}\|_{L^{2}}^{2}\big)
+ C_{8}\B(\int |\nabla u|^{2}\B) \B(\int  |\text{div\,}u|^3 \B) ^{\f23}.
 \ea\ee
\eqref{cond2} and Gronwall allow us to obtain  Lemma \ref{lemma3.3} without $\lambda<3\mu$.

  Case 2:  \be
  \| \text{div\,}u\|_{L^{1}(0,T;L^{\infty})}
  +\|\text{div\,}u\|_{L^{4}(0,T;L^{2})}
  +\|\theta\|_{L^{p}(0,T;L^{q})}\leq
  C.\label{cond3}
  \ee
  From the above arguments of Case 1,
   we just need prove the following estimate
  \be\ba\label{34.8}
 \int |u|^{2}|\text{div\,}u|^2
 &\leq C\B(\int |u|^{6}\B) ^{\f13}\B(\int  |\text{div\,}u|^3 \B) ^{\f23}\\
  &\leq C\B(\int |\nabla u|^{2}\B) \B(\int  |\text{div\,}u|^3 \B) ^{\f23}\\
  &\leq C\B(\int |\nabla u|^{2}\B) \B(\int  |\text{div\,}u|^2 \B) ^{\f23}\|\text{div\,}u\|_{L^{\infty}} ^{\f23}\\
  &\leq C\B(\int |\nabla u|^{2}\B) \B(\B(\int  |\text{div\,}u|^2 \B) ^{2}+\|\text{div\,}u\|_{L^{\infty}} \B),
 \ea\ee
where we have used
   the H\"older inequality, Sobolev embedding  and interpolation inequality.

   Case 3:  \be\label{cond4}
  \| \text{div\,}u\|_{L^{2}(0,T;L^{\infty})} +\|\theta\|_{L^{p}(0,T;L^{q})}\leq C.\ee
  From \eqref{eqofE}, we have
  \be\label{bofrctia}
  \|\rho\theta\|_{L^{\infty}(0,T;L^{1})}\leq C.
  \ee
  Taking the $L^{2}$ inner product of the second equations in \eqref{FNS} with $u$, we see that
  \be\ba\label{34.9}
\f{d}{dt}\int\rho|u|^{2}+\mu\int|\nabla u|^{2}+(\lambda+\mu)\int|\text{div\,} u|^{2}  \leq C \|\text{ div\,} u\|_{L^{\infty}}\|\rho\theta\|_{ L^{1} }
\ea\ee
  It follows from \eqref{bofrctia} and \eqref{34.9}
  that
  \be\label{intiesti}
  \int_{0}^{T}\int|\text{div\,} u|^{2}dx\leq C.
  \ee
  From the above arguments of Case 1 and Case 2, we just need prove the following estimate
\be\ba\label{3.4.15}
 \int |u|^{2}|\text{div\,}u|^2
 &\leq C\B(\int |u|^{6}\B) ^{\f13}\B(\int  |\text{div\,}u|^3 \B) ^{\f23}\\
  &\leq C\B(\int |\nabla u|^{2}\B) \B(\int  |\text{div\,}u|^3 \B) ^{\f23}\\
  &\leq C\B(\int |\nabla u|^{2}\B) \B(\int  |\text{div\,}u|^2 \B) ^{\f23}\|\text{div\,}u\|_{L^{\infty}} ^{\f23}\\
  &\leq C\B(\int |\nabla u|^{2}\B) \B(\B(\int  |\text{div\,}u|^2 \B) +\|\text{div\,}u\|^{2}_{L^{\infty}} \B).
 \ea\ee
 As the same  derivation of \eqref{34.7}, replacing \eqref{34.6} by \eqref{3.4.15}, we get
\be\ba\label{34.16}
& \f{d}{dt}\int\B[\f\mu2|\nabla u|^{2}+(\mu+\lambda)(\text{div\,}u  )^2+\f{1}{2\mu+\lambda}P^2 -2P\text{div\,}u +\f{C_3 C_\nu}{2}\rho\theta^2 \\
&\ \ +\f{C_4 +1}{2\mu}\rho| u|^4\B]+ \kappa\int|\nabla \theta|^{2}+\f{1}{2}\int \rho |\dot{u}|^{2}+2\mu\int|u|^{2}|\nabla
u|^2\\\leq& C \|\theta\|^{\f{2q}{2q-3}}_{L^{q}}\big(\|\rho^{\f12}\theta\|_{L^{2}}^{2}+\|\rho^{\f12}|u|^{2}\|_{L^{2}}^{2}\big)
+  C\B(\int |\nabla u|^{2}\B) \B(\B(\int  |\text{div\,}u|^2 \B) +\|\text{div\,}u\|^{2}_{L^{\infty}} \B).
  \ea\ee
Gronwall lemma
\eqref{intiesti}, \eqref{34.16}, and \eqref{cond4} yield Lemma \ref{lemma3.3} without $\lambda<3\mu$.
\end{proof}

 \section{Blow up criteria without vacuum}
As said in the last of subsection 3.1, it is enough to show
  Lemma \ref{lemma3.3} without $\lambda<3\mu$ to complete  the proof of Theorem \ref{th1.3} under the following hypothesis
 \be
  \|\rho, \ \rho^{-1}|_{L^{\infty}(0,T;L^{\infty})}+
  \|\text{div\,}u\|_{L^{p_{1}}(0,T;L^{q_{1}})}
  +\|\theta\|_{L^{p_{2}}(0,T;L^{q_{2}})}\leq C,\label{cond5}\ee
  for  $(p_{1},\,q_{1})$ and $(p_{2},\,q_{2})$ meeting
   \be\label{cond51}\f{2}{p_{1}}+\f{3}{q_{1}}=2,~ q_{1}>\f{3}{2},~~\f{2}{p_{2}}+\f{3}{q_{2}}=2, ~   q_{2}>\f{3}{2}. \ee
\begin{proof}[Proof of Lemma \ref{lemma3.3} without $\lambda<3\mu$]
From \eqref{34.1} and \eqref{34.2}, we see that
\be  \ba\label{4.2}
&\frac{d}{dt}\int\
\rho|u|^4+2\mu\int  |u|^{2} |\nabla
u|^2\leq C\int\rho |u|^{2}|\theta|^{2} +C\int |\text{div\,}u||u|^{2}|\nabla|u|.
\ea\ee
The interpolation inequality and Sobolev inequality allow us
to derive that, for  $2\leq\f{2q_{1}}{q_{1}-2}\leq6$,
\be\ba\label{4.3}
\|\rho^{\f{1}{2}}|u|^{2}\|_{L^{\f{2q_{1}}{q_{1}-2}}}&\leq
\|\rho^{\f{1}{2}}|u|^{2}\|^{1-\f{3}{q_{1}}}_{L^{2}}\|\rho^{\f{1}{2}}|u|^{2}\|^{\f{3}{q_{1}}}_{L^{6}}\\
&\leq C
\|\rho^{\f{1}{2}}|u|^{2}\|^{1-\f{3}{q_{1}}}_{L^{2}}
\| |u|^{2}\|^{\f{3}{q_{1}}}_{L^{6}}\\
&\leq C
\|\rho^{\f{1}{2}}|u|^{2}\|^{1-\f{3}{q_{1}}}_{L^{2}}
\| \nabla|u|^{2}\|^{\f{3}{q_{1}}}_{L^{2}}.
\ea\ee
In the light of the H\"older inequality, \eqref{4.3}, and the Young inequality, we find
 \be\ba\label{4.5}
\int |\text{div\,}u|u|^{2}|\nabla|u|&\leq C \int|\text{div\,}u|\rho^{\f{1}{2}}|u|^{2} | |\nabla u|\\
&\leq C\|\text{div\,}u\|_{L^{q_{1}}}\|\rho^{\f{1}{2}}|u|^{2}\|_{L^{\f{2q_{1}}{q_{1}-2}}}\|\nabla u\|_{L^{2}}\\
&\leq C(\eta)\|\text{div\,}u\|_{L^{q_{1}}}\|\rho^{\f{1}{2}}|u|^{2}\|^{1-\f{3}{q_{1}}}_{L^{2}}\|\rho^{\f{1}{2}}|u|^{2}\|^{\f{3}{q_{1}}}_{L^{6}}\|\nabla u\|_{L^{2}}\\
&\leq \eta\|\nabla|u|^{2}\|^{2}_{L^{2}}+C(\eta)\|\text{div\,}u\|^{\f{2q_1}{2q_1-3}}_{L^{q_1}}\|\rho^{\f{1}{2}}|u|^{2}\|^{ \f{2(q_{1}-3)}{2q_{1}-3}}_{L^{2}}\|\nabla u\|_{L^{2}}^{\f{2q_{1}}{2q_{1}-3}}\\
&\leq \eta\|\nabla|u|^{2}\|^{2}_{L^{2}}+
C(\eta)\|\text{div\,}u\|^{\f{2q_{1}}{2q_{1}-3}}_{L^{q_{1}}}
\B(\|\rho^{\f{1}{2}}|u|^{2}\|^{ 2}_{L^{2}}+\|\nabla u\|_{L^{2}}^{2}\B).
\ea\ee
It follows from \eqref{4.2} and \eqref{4.5} that
 \be  \ba \label{key5}
\frac{d}{dt}&\int\
\rho|u|^4+2\mu\int|u|^{2}|\nabla
u|^2\\
&\leq C_{8}\int\rho |u|^{2}|\theta|^{2}
+C_{8}\B(\int\rho|u|^{4}+\int  |\nabla u|^2 \B) \|\text{div\,}u\|^{\f{2q_{1}}{2q_{1}-3}}_{L^{q_{1}}}.
  \ea\ee

Adding  $\eqref{key5}$ multiplied by  $ \f{C_{4}+1}{2\mu}$ to $  \eqref{key3}$, we arrive at
 \be\ba
&\f{d}{dt}\int\B[\f\mu2|\nabla u|^{2}+(\mu+\lambda)(\text{div\,}u  )^2+\f{1}{2\mu+\lambda}P^2 -2P\text{div\,}u +\f{C_3 C_\nu}{2}\rho\theta^2 \\
&\ \ +\f{C_4 +1}{2\mu}\rho| u|^4\B]+ \kappa\int|\nabla \theta|^{2}+\f{1}{2}\int \rho |\dot{u}|^{2}+2\mu\int|u|^{2}|\nabla
u|^2 \\&\leq C  \int \rho^{2}|\theta|^{3}  +C \int\rho |u|^{2}|\theta|^{2}
+C \B(\int\rho|u|^{4}+\int  |\nabla u|^2 \B) \|\text{div\,}u\|^{\f{2q_1}{2q_1-3}}_{L^{q_1}} .
 \ea\ee
 Then, we use \eqref{3.42} and \eqref{3.43} to further obtain
  \be\ba
\f{d}{dt}\int& \B[\f\mu2|\nabla u|^{2}+\rho\theta^{2}
+\rho|u|^4
\B]+ \kappa\int|\nabla \theta|^{2}+\f{1}{2}\int \rho |\dot{u}|^{2}+2\mu\int|u|^{2}|\nabla
u|^2\\&\leq C \|\theta\|^{\f{2q_{2}}{2q_{2}-3}}_{L^{q_{2}}}\big(\|\rho^{\f12}\theta\|_{L^{2}}^{2}+\|\rho^{\f12}|u|^{2}\|_{L^{2}}^{2}\big)
+ C \B(\int\rho|u|^{4}+\int  |\nabla u|^2 \B) \|\text{div\,}u\|^{\f{2q_{1}}{2q_{1}-3}}_{L^{q_{1}}}.
\ea\ee
This proves the whole lemma.
\end{proof}

		\section*{Acknowledgement}
Jiu was partially
supported by the National Natural Science Foundation of China (No. 11671273) and by Beijing Natural
Science Foundation (No. 1192001).
  Wang was partially supported by  the National Natural
Science Foundation of China under grant (No. 11971446  and No. 11601492). Ye is supported by the NSFC (No. 11701145).

	\end{document}